\documentclass[11pt, english]{article}
\usepackage{preamble}
\usepackage{subcaption}

\title{On the approximation of L\'evy driven Volterra processes and their integrals}

\author{Giulia Di Nunno\thanks{Department of Mathematics,
University of Oslo, P.O. Box 1053 Blindern, N-0316 Oslo, Email: giulian@math.uio.no.} \thanks{Department of Business and Management Science, NHH Norwegian School of Economics, Helleveien 30, N-5045 Bergen.}
 \and Andrea Fiacco\thanks{Department of Mathematics, University of Oslo, P.O. Box 1053 Blindern, N-0316 Oslo. Email andrefi@math.uio.no.}  \and
Erik Hove Karlsen\thanks{IF Skadeforsikring, Drammensveien 264, 0283 Oslo. Email: erik.h.karlsen@gmail.com}}
\date{December 21st, 2018}

\begin{document}
\maketitle
\begin{abstract}
Volterra processes appear in several applications ranging from turbulence to energy finance where they are used in the modelling of e.g. temperatures and wind and the related financial derivatives. Volterra processes are in general non-semimartingales and a theory of integration with respect to such processes is in fact not standard. In this work we suggest to construct an approximating sequence of L\'evy driven Volterra processes, by perturbation of the kernel function. In this way, one can obtain an approximating sequence of semimartingales.

Then we consider fractional integration with respect to Volterra processes as integrators and we study the corresponding approximations of the fractional integrals.
We illustrate the approach presenting the specific study of the Gamma-Volterra processes. Examples and illustrations via simulation are given.

  \vspace{2mm}\noindent
  {\bf Keywords:}
Riemann-Liouville fractional integral, Volterra processes, fractional Brownian motion, ambit processes, generalized Lebesgue-Stieltjes integral, non-semimartingales.
\end{abstract}


\section{Introduction}
We consider Volterra type processes driven by L\'evy noise $L(t),\;t\geq0$, of the form:
\begin{equation}\label{eq:mainint}
 Y(t):=\int_0^tg(t-s)dL(s),\quad t\geq0,
\end{equation}
where $g$ is a deterministic kernel. Such processes appear in many different applications including models for tumour growth, turbulence, and energy finance, see e.g. \cite{BNS,veraart2,veraart3,S}. Processes of type \eqref{eq:mainint} belong to the family of ambit fields as presented e.g. in \cite{veraart} and include, as particular cases, the L\'evy fractional Brownian motion given by the Riemann-Liouville integral, see \cite{MVN}. The fractional Brownian motion is represented (modulo a constant factor) by an integral of type \eqref{eq:mainint} plus a suitable process with absolutely continuous trajectories, see \cite[p. 424]{MVN}. Compare also with the integral representation on $(0,t]$ with the Molchan-Golosov kernel, see e.g. \cite{Jost}. For fractional L\'evy processes we can refer e.g. to \cite{Marquardt2006, BM2008, bender} and references therein.

In general Volterra processes are not semimartingales, see \cite{basse}. We recall that semimartingales constitute the largest class of integrators for a stochastic integration theory (It\^o-type integration) which is well-suited for applications where the adaptedness or the predictability with respect to a given information flow plays an important role. 
This is the case, for example, in mathematical finance where one needs integration to define e.g. the central concept of the value process of a self financing portfolio. Also, the numerical methods are flourishing in the case of semimartingale models. Without means of being exhaustive, we can refer, e.g., to classical books \cite{Bouleau-Lepingle-book, KP} and to more recent works that show that the area is in simmering activity \cite{Platen-book2, song-wang, Wang, Yan }. 
Processes of type \eqref{eq:mainint} have interesting stylized features, like the non-trivial time correlation structure, that well suits several contexts of modelling, such as in renewable energies. In energy finance the use of non-semimartingale models is well motivated. See e.g. \cite[section 3.3]{veraart2} for a discussion.

In this paper we propose to approximate \eqref{eq:mainint} by the process
\begin{equation}\label{eq:mainintapprox}
 Y^\e(t):=\int_0^tg^\e(t-s)dL(s),
\end{equation}
where \( g^\e\), with $\e\in(0,1)$, is a family of deterministic kernel functions approximating $g$, i.e. $g^\e\longrightarrow g$ as $\e\rightarrow0$, in an appropriate sense. We are interested in the cases when $g^\e$ guarantees that $Y^\e(t),\;t\geq0$, is a semimartingale and we show that $Y^\e(t)$ approximates $Y(t)$ in the sense of $L_p$-convergence.

Approximations of this type were first introduced in \cite{thao} and \cite{thao2}, and then used in \cite{MR2782221}, but only in the case where $Y$ is a fractional Brownian motion. Our result extends substantially this first study and moves beyond. 

In fact, the core of the present paper deals with the generalized Lebesgue-Stieltjes integrals with respect to the processes \eqref{eq:mainint} and \eqref{eq:mainintapprox} as integrators. This is a form of pathwise integration defined via the fractional derivatives. For a survey, new results and conditions for integration with respect to Volterra type processes as integrators see \cite{dinunno}. In this study we suggest sufficient conditions to ensure that, for a given integrand $X$, the generalized Lebesgue-Stieltjes integrals with respect to $Y^\e$ and $Y$ as integrators converge in $L_1$:
\begin{equation}\label{eq:gLSapprox}
 \int_0^T X(s)dY^\e(s)\longrightarrow\int_0^T X(s)dY(s),\quad\e\rightarrow0.
\end{equation}

We remark that, if $Y^\e$ is a semimartingale and $X$ is a predictable process (with respect to the same filtration), the generalized Lebesgue-Stieltjes integral corresponds to the It\^o type integral. Hence, in the context of predictable integrands, the approximation \eqref{eq:gLSapprox} provides an approximation of a non-semimartingale by a semimartingale. We intend to exploit this feature in future research dealing with hedging in energy finance.
Here we illustrate the use of the approximation in simulation with an example.

We illustrate the results in full detail in the case of
\begin{equation}\label{eq:keyex}
 Y(t):=\int_0^t(t-s)^\beta e^{-\lambda(t-s)}dL(s),
\end{equation}
for $\beta\in(-1/2,1/2)$, $\lambda\geq0$. In this case \( g \) is, up to a constant, a Gamma kernel. For $\beta\in(-1/2,0)$, the integral \eqref{eq:keyex} is obtained as an appropriate stochastic modification of the Riemann-Liouville fractional integral in which the factor $e^{-\lambda(t-s)}$ in the kernel has a dampening effect. The processes \eqref{eq:keyex} appear explicitly in the modelling of turbulence and in the modelling of environmental risk factors in energy finance (e.g. wind), see \cite{BN,VK}. In the sequel we refer to \eqref{eq:keyex} as Gamma-Volterra process.
In view of the relevance of this family in applications, we shall detail the study of such processes.

The paper is organised as follows. The next section reviews knowledge about Volterra processes and introduces an approximation by perturbation of the kernel. Particularly interesting is the case when the Volterra process is not a semimartingale and it can be approximated by a semimartingale process. As illustration, the L\'evy driven Gamma-Volterra processes are studied along with their approximations.
Section 3 deals with fractional integration and it is divided in two parts. In the first part we revise general facts and then we provide conditions to guarantee when a Volterra process is an appropriate integrator. This includes cases when the Volterra process is not a semimartingale. Examples are provided.
In the second part of the section, exploiting the approximation introduced before, we suggest an approximation of the integral with respect to a Volterra process. Examples and full detailed conditions are provided in the case of a Gamma-Volterra process. Finally, a numerical example is given as direct application and illustration of the technique proposed.


\section{Volterra processes and a semimartingale \\approximation}
First of all we review the fundamental concepts to ensure the meaningful definition of \( Y \) in \eqref{eq:mainint}. We define the integration of a deterministic function with respect to the L\'evy process \( L \) as in \cite{dinunno} by the approach proposed in \cite{urbanik} and further developed in \cite{rajput}.

Let \( (\Omega, \F,\prob) \) be a complete probability space and \( L=L(t),\;t\geq0, \) be a L\'evy process with characteristic function represented in the following form (see e.g. \cite{sato}):
\begin{equation*}
 \E\left[e^{ixL(t)}\right]=e^{t\psi(x)}, \quad x \in \mathbb{R},
\end{equation*}
with
\begin{equation*}
 \psi(x)=iax-\frac{x^2b}{2}+\int_\R\{e^{ixz}-1-ix\tau(z)\}\nu(dz),
\end{equation*}
where
\begin{equation*}
 \tau(z):=
 \begin{cases}
  z,\quad|z|\leq1\\
  \frac{z}{|z|},\quad|z|>1,
 \end{cases}
\end{equation*}
\( a\in\R,\;b\geq0\), and \( \nu \) is a L\'evy measure on \( \R \), i.e. it is a \( \sigma \)-finite Borel measure satisfying
\begin{equation*}
 \int_\R (z^2\wedge 1)\nu(dz)<\infty,\quad\nu(\{0\})=0.
\end{equation*}
The triplet \( (a,b,\nu) \) is called the characteristic triplet of the L\'evy process \( L \).

\vspace{2mm}

From the increments \( L((s,t]):=L(t)-L(s),\;s\leq t \) of the L\'evy process $L$, we obtain the random measure on \( \borel([0,\infty)) \) taking values in \( L_0(\Omega,\F,\prob) \), see \cite{rajput}. The random measure is still denoted by $L$.  
For any \( A\in\borel([0,\infty)) \) s.t. \( \lambda(A)<\infty \), the random measure values \( L(A) \) are random variables with infinitely divisible distribution and L\'evy-Khintchine characteristic function
\begin{equation*}
 \E\left[e^{ixL(A)}\right]=e^{\lambda(A)\psi(x)}, \quad x \in \mathbb{R}.
\end{equation*}
Here $\lambda$ denotes the Lebesgue measure on $\borel(\R)$.

\begin{definition}  \hspace{2cm}
\label{integrand}
 \begin{enumerate}[(i)]
  \item Let $f=\sum_{j=1}^J f_j\mathbf{1}_{A_j}$ be a real-valued simple function on $[0,T]$, where the pairwise disjoint sets $A_j\in\borel([0,T])$ belong to a partition of $[0,T]$. Then, for any $A\in\borel([0,T])$, we set
  \[
   \int_AfdL:=\sum_{j=1}^J f_jL(A\cap A_j).
  \]
 \item A measurable function $f:([0,T],\borel([0,T]))\longrightarrow(\R,\borel(\R))$ is said to be $L$-integrable (on $[0,T]$) if there exists a sequence $\{f_n\}_{n\geq 1}$ of simple functions as in (i) such that
 \begin{enumerate}
  \item $\lim_{n\to \infty } f_n = f $, $ \lambda$-$a.e.$
  \item for any $A\in\borel([0,T])$, the corresponding sequence $\{\int_A f_n dL\}_{n\geq 1}$ converges in probability as $n\rightarrow\infty$.
 \end{enumerate}
 \end{enumerate}
 If $f$ is $L$-integrable, the stochastic integral on $A\in \borel([0,T])$ is defined by
 \[
  \int_AfdL:=\lim_{n\rightarrow\infty}\int_Af_ndL,
 \]
with convergence in probability.
\end{definition}

\noindent 
The integral is well-defined, i.e. for any \( L \)-integrable function \\
\( f:([0,T],\borel([0,T])) \longrightarrow (\R,\borel(\R)) \), the integral does not depend on the choice of approximating sequence \( \{f_n\}_{n \geq 1} \).
Moreover, the integral $\int_A f dL$ is also infinitely divisible with explicit characteristic function, see \cite{rajput, urbanik}. 

\vspace{3mm}
\noindent 
The following result characterizes the space of integrands. See e.g. Lemma 2.1 in \cite{dinunno}.
\begin{lemma}\hspace{2cm}
\label{L-integrands}
\begin{enumerate}[(i)]
\item For \( p\geq2 \), any function \( f\in L_p([0,T]) \) is \( L \)-integrable.
\item For \( p\in[1,2) \) assume that \(L\) satisfies \( b=0 \) and \( \int_{|z|\leq1}|z|^p\nu(dz)<\infty \). 
Then any function \( f\in L_p([0,T]) \) is \( L \)-integrable.
\end{enumerate}
\end{lemma}

\noindent
Hence for all $t$, under the conditions of Lemma \ref{L-integrands}, we have that the integral \eqref{eq:mainint} is well defined for $L$-integrable functions $g(t-\cdot)$  on $[0,t]$.
The proper definition of $Y$ is a standing assumption in this work.

\vspace{3mm}
Depending on the properties of the kernel function \( g \), the Volterra process may or may not be a semimartingale. The semimartingale property of various subclasses of Volterra type processes is studied in e.g. \cite{basse, rosinski, bender, knight}. Hereafter, we fix the natural filtration \( \mathbb{F}=\{ \F_t,\;t \geq 0\} \) generated by the L\'evy process \( L \) with the characteristic triplet \( (a,b,\nu) \) on \( (\Omega,\F,\prob) \) and we state the necessary and sufficient conditions to guarantee that the Volterra process \( Y \)  in \eqref{eq:mainint} is a semimartingale. See \cite{basse}, Theorem 3.1 and Corollary 3.5.

\begin{theorem}\label{thm:basse}
Assume that $L(t),\;t\geq0,$ is of unbounded variation. Then $Y$ is an $\mathbb{F}$-semimartingale if and only if $g=g(u)$, $u\geq0$, is absolutely continuous on $\R_+$ with a density $g'$ such that
$$
\int_0^t \vert g'(u)\vert^2 du < \infty, \quad t \geq 0,
$$
when $b>0$, and satisfies
 \begin{equation}\label{eq:semicond}
  \int_0^t\int_{[-1,1]}|zg'(u)|^2\wedge|zg'(u)|\nu(dz)du<\infty,\quad t>0,
 \end{equation}
 when $b=0$.
 
Assume that $L(t),\;t\geq0,$ is of bounded variation. Then $Y(t),\;t\geq0,$ is an $\mathbb{F}$-semimartingale if and only if it is of bounded variation, which is equivalent to requesting that $g$ is of bounded variation.
\end{theorem}

\vspace{2mm}
\begin{example}\label{coro:semi}
{\bf Semimartingale property of Gamma-Volterra processes.}
Consider the Gamma-Volterra process \eqref{eq:keyex}:
\begin{equation*}
 Y(t):=\int_0^t(t-s)^\beta e^{-\lambda(t-s)}dL(s), \quad t \geq0,
\end{equation*}
with $\beta \ne 0, \lambda \geq0$.
From direct application of the theorem above we see that if $L$ is a Brownian motion or a L\'evy process with $b>0$, then $Y$ is a $\mathbb{F}$-semimartingale if and only if $\beta>1/2$.
If $L$ is a L\'evy process with no Brownian component, i.e. \( b=0 \), then $Y$ is well-defined and an $\mathbb{F}$-semimartingale if and only if one of the following conditions is satisfied (see \cite{basse}, Corollary 3.5):
 \begin{enumerate}[(i)]
  \item $\beta>1/2$,
  \item $\beta=1/2$ and $\int_{[-1,1]}z^2|\log|z||\nu(dz)<\infty$,
  \item $\beta\in(0,1/2)$ and $\int_{[-1,1]}z^{1/(1-\beta)}\nu(dz)<\infty$.
 \end{enumerate}
\end{example}

\vspace{3mm}
The following result is a moment estimate for the L\'evy driven Volterra processes, see Theorem 2.2 and Remark 2.2 in \cite{dinunno}. This is obtained under the technical assumption that $\nu$ is symmetric. We shall make this assumption in our present work.
\begin{theorem}\label{thm:dinunno}
Let $L=L(t),\;t\geq0$, be a L\'evy process with symmetric L\'evy measure $\nu$. We have the following two statements:
 \begin{enumerate}[(a)]
  \item  For a L\'evy process with characteristic triplet $(a,0,\nu)$ such that $\int_\R|z|^p\nu(dz)<\infty$ for some $p\geq1$, we assume that for \( t \geq0 \), \( g(t,\cdot)\in L_p([0,t]) \). Then $g(t,\cdot)$ is $L-$integrable and we have the estimate:
  \begin{equation*}\label{eq:est1}
   \E\left|\int_0^tg(t,s)dL(s) \right|^p\leq C_1\left(|a|^p\|g(t,\cdot)\|^p_{L_1[0,t]}+\|g(t,\cdot)\|^p_{L_p[0,t]}\int_\R|z|^p\nu(dz)\right).
  \end{equation*}
  \item For a L\'evy process with characteristic triplet $(a,b,\nu)$ such that $\int_\R|z|^p\nu(dz)<\infty$ for some $p\geq2$, we assume that for \( t\geq 0 \), \( g(t,\cdot)\in L_p([0,t]) \). Then $g(t,\cdot)$ is $L-$integrable and we have the estimate:
  \begin{align*}\label{eq:est2}
   \E\left|\int_0^tg(t,s)dL(s) \right|^p\leq C_2\bigg(&|a|^p\|g(t,\cdot)\|^p_{L_1[0,t]}+b^{p/2}\|g(t,\cdot)\|^p_{L_2[0,t]}\\
   &+\|g(t,\cdot)\|^p_{L_p[0,t]}\int_\R|z|^p\nu(dz)\bigg).
  \end{align*}
 \end{enumerate}
 The constants $C_1,C_2$ do not depend on $g$.
\end{theorem}
Notice that, in the present work, all constants in the estimates are denoted by $C$. Their dependence on the parameters can be explicitly given when relevant. Their specific form is deduced from the context.

\vspace{3mm}
\noindent\textbf{Remark.} Recall that a L\'evy process with characteristic triplet $(0,b,\nu)$ is a square-integrable martingale if and only if, for some $p\geq2$,
\begin{equation*}
\int_{|z| \geq 1} |z|^p \; \nu(dz)< \infty \quad \textit{and} \quad \int_{|z| \geq 1} z \; \nu(dz)=0.
\end{equation*}
Then, considering the assumptions of Theorem \ref{thm:dinunno} in this case, if the L\'evy process $L$ has symmetric L\'evy measure $\nu$ and $\int_{\mathbb{R}} |z|^p \; \nu(dz)< \infty$ for some $p\geq2$, this L\'evy process is a square-integrable martingale with $\langle L \rangle_t =t\,  (b + \int_{\mathbb{R}} z^2 \; \nu(dz))$, $t\geq0$. In this case we could also consider It\^o stochastic integration of predictable stochastic processes $g(t-\cdot)$ and find estimates of the moments based on the Burkholder-Davis-Gundy and Bichteler-Jacod types inequalities. See e.g. Lemma 5.1 in \cite{Jacod-et-al}.  However, we shall not consider such processes in the framework of the present work. We remark that other similar estimates can be found by means of Rosenthal inequalities in the case of Poisson stochastic integrals, which are optimal in the sense that an It\^o isomorphism is obtained, see \cite{dirksen}.

\vspace{3mm}
For later use, we consider the following result.
\begin{lemma}\label{lemma:leftlimit}
Let $L$ be a L\'evy process with characteristic triplet $(a,b,\nu)$, where $\nu$ is symmetric and $\int_{\mathbb{R}} |z|^p \; \nu(dz)< \infty$ for some $p\geq1$, if $b=0$, or some $p\geq2$, if $b>0$. Then $g(t-\delta-\cdot)\mathbf{1}_{(0,t-\delta)}(\cdot)\in L_p([0,t])$, for all $\delta \in (0,t)$, and the integrals
$$
Y(t-\delta) = \int_0^{t-\delta} g(t-\delta-s)dL(s), \quad \delta \in (0,t),
$$
are well defined.
\noindent
Assume that $\lim_{\delta \downarrow 0} g(t-\delta-\cdot)$ exists with convergence in $L_p([0,t])$ and denote by $g(t^--\cdot)$ the limit and the corresponding function defined $s$-a.e. by a subsequence. Then the integral 
$$
Y(t^-) := \int_0^t g(t^- -s)dL(s)
$$
is well-defined. Furthermore, $\lim_{\delta \downarrow 0} Y(t-\delta)$ exists with convergence in $L_p(\Omega)$ (and in probability) and $Y(t^-) = \lim_{\delta \downarrow 0} Y(t-\delta)$.
\end{lemma}
\begin{proof}
Since $g(t-\delta-\cdot)\mathbf{1}_{(0,t-\delta)}(\cdot)\in L_p([0,t])$, by convergence, also $g(t^--\cdot)\in L_p([0,t])$, then the corresponding integrals $Y(t-\delta)$ and $Y(t^-)$ are well-defined by Lemma \ref{L-integrands}. We prove the last assertion. It is enough to show that the sequence $(Y(t-\delta))_\delta$ admits a limit in $L_p(\Omega)$.

\noindent
Applying the estimates of Theorem \ref{thm:dinunno} we can see that, for $\delta,\rho>0$ small enough,
\begin{align*}
\mathbb{E}|Y(t-\delta) &- Y(t-\rho)|^p =\\
&= \mathbb{E}\Bigl| \int_0^t [g(t-\delta-s)\mathbf{1}_{(0,t-\delta)}(s) - g(t-\rho-s)\mathbf{1}_{(0,t-\rho)}(s)]dL(s)\Bigr|^p \\
&\leq C_{a,b,\nu}\|g(t-\delta-\cdot)\mathbf{1}_{(0,t-\delta)} - g(t-\rho-\cdot)\mathbf{1}_{(0,t-\rho)}\|_{L_p[0,t]}^p \\
& \quad \longrightarrow 0, \quad for \; \delta,\rho \rightarrow 0.
\end{align*}
Thus the sequence $(Y(t-\delta))_\delta$ is Cauchy in $L_p(\Omega)$.
Analogously, we see that
\begin{align*}
\mathbb{E}|Y(t-\delta) - Y(t^-)|^p &= \mathbb{E}\Bigl| \int_0^t [g(t-\delta-s)\mathbf{1}_{(0,t-\delta)}(s) - g(t^--s)]dL(s)\Bigr|^p \\
&\leq \tilde{C}_{a,b,\nu}\|g(t-\delta-\cdot)\mathbf{1}_{(0,t-\delta)} - g(t^--\cdot)\|_{L_p[0,t]}^p \\
& \quad \longrightarrow 0, \quad for \; \delta \rightarrow 0.
\end{align*}
\end{proof}

Hereafter we study an approximation for the Volterra process $Y=Y(t),\;t\geq 0$, derived by perturbation of the kernel function. 
Let \( g^\e,\;\e\in(0,1) \), be a family of deterministic $L$-integrable kernels and define the corresponding family of Volterra processes  $Y^\e=Y^\e(t),\;t\geq 0$, by 
\begin{equation}
\label{eq:approxx}
 Y^\e(t)=\int_0^t g^\e(t-s)dL(s), \quad t \geq 0.
 \end{equation}
 
\begin{theorem}\label{coro:approx1}
 Let $L=L(t),\;t\geq0$, be a L\'evy process with symmetric L\'evy measure. Consider one of the following situations:
 \begin{enumerate}[(a)]
  \item The L\'evy process has characteristic triplet $(a,0,\nu)$ and $\int_\R|z|^p\nu(dz)<\infty$ for some $p\geq1$.
  \item The L\'evy process has characteristic triplet $(a,b,\nu)$ and $\int_\R|z|^p\nu(dz)<\infty$ for some $p\geq2$.
 \end{enumerate}
 Then, for any $t$, we have the convergence in $L_p(\Omega)$:
 \begin{equation}\label{eq:asterix}
  \|Y^\e(t)-Y(t)\|_{L_p(\Omega)}\longrightarrow0,\quad\text{ as }\e\rightarrow0,
 \end{equation}
 whenever $g(t -\cdot),g^\e(t-\cdot)\in L_p[0,t]$ such that
  \begin{equation}\label{eq:asterix2}
    \|g^\e(t-\cdot)-g(t-\cdot)\|_{L_p[0,t]}\longrightarrow0,\quad\text{ as }\e\rightarrow0.
  \end{equation}
  If \eqref{eq:asterix2} is uniform on \( t\in[0,T] \) $(T<\infty)$, then \eqref{eq:asterix} would be uniform on \( t\in[0,T] \) as well.
\end{theorem}

\begin{proof}
Fix $t \geq 0$.
 Consider case (a). By Theorem \ref{thm:dinunno}(a) there exists some \( C>0 \) such that
 \begin{align*}
  \E\left|Y^\e(t)-Y(t)\right|^p&=\E\left|\int_0^tg^\e(t-s)-g(t-s)dL(s) \right|^p\\
  &\leq C\Big(|a|^p\|g^\e(t-\cdot)-g(t-\cdot)\|^p_{L_1[0,t]}\\
  &\quad\quad\quad+\|g^\e(t-\cdot)-g(t-\cdot)\|^p_{L_p[0,t]}\int_\R|z|^p\nu(dz)\Big)\longrightarrow0,
 \end{align*}
 as $\e\rightarrow0$. Similarly, for the convergence in (b) we apply Theorem \ref{thm:dinunno}(b) and there exists some \( C>0 \) such that
 \begin{align*}
  \E\left|Y^\e(t)-Y(t)\right|^p&=\E\left|\int_0^tg^\e(t-s)-g(t-s)dL(s) \right|^p\\
  &\leq C\Big(|a|^p\|g^\e(t-\cdot)-g(t-\cdot)\|^p_{L_1[0,t]}\\
  &\quad\quad\quad+b^{p/2}\|g^\e(t-\cdot)-g(t-\cdot)\|^p_{L_2[0,t]}\\
  &\quad\quad\quad+\|g^\e(t-\cdot)-g(t-\cdot)\|^p_{L_p[0,t]}\int_\R|z|^p\nu(dz)\Big)\longrightarrow0,
 \end{align*}
 as $\e\rightarrow0$.
\end{proof}

\vspace{3mm}
\noindent In the following example we specify under which assumptions we can approximate the Gamma-Volterra process in~\eqref{eq:keyex} with a semimartingale, using Theorem \ref{coro:approx1}.

\begin{example}\label{ex:Y}
{\bf Approximation of Gamma-Volterra processes.}
Let $Y(t),\;t\geq0$, be the Gamma-Volterra process in~\eqref{eq:keyex} with driving noise $L$, a L\'evy process with the characteristic triplet \( (a,0,\nu) \), where \( \nu \) is a symmetric measure such that $\int_\R|z|^p\nu(dz)<\infty$, for some $p\geq1$.
Fix $t$.
From Lemma \ref{L-integrands} we have that \eqref{eq:keyex} is well defined whenever $g(t-\cdot) \in L_p[0,t]$. That is
 \begin{align*}
  \int_0^t(t-s)^{\beta p}e^{-\lambda p(t-s)}ds\leq\int_0^t(t-s)^{\beta p}ds<\infty,
 \end{align*}
whenever $\beta p +1 > 0$. We shall consider two cases:
\begin{enumerate}
\item[(a)]
$ \lambda \geq 0, \beta \in (-1,0)$ and $\beta p +1 >0$, i.e. $p \in [1, 1/\vert \beta \vert)$;
\item[(b)]
$\lambda \geq 0, \beta > 0$;
\end{enumerate}
Concerning the approximating process $Y^\e$ in \eqref{eq:approxx}, we define 
$$
 g^\e(u) :=(u+\e)^\beta e^{-\lambda (u+\e)} , \qquad u \in (0,t], \; \e \in (0,1\wedge t).
 $$
 Correspondingly, we have
 $$
 Y^\e(t) := \int_0^t g^\e(t-s) dL(s),\;t\geq 0.
 $$
 These processes are well-defined and, applying Theorem \ref{thm:basse}, we can see that $Y^\e$ are semimartingales, since \( g^\e \) are absolutely continuous on \( \R_+ \) and \( (g^\e)' \) is bounded on $[0,t]$ for $0 \leq t <\infty$.
 
 We can also see that $g^\e(t-\cdot)-g(t-\cdot)\in L_p[0,t]$ since $g^\e$ is bounded. Hereafter we give an estimate of this difference and we distinguish the cases in which $\beta$ is positive or negative.

\vspace{2mm}\noindent
(a)$\quad$ We consider the case $\beta  \in (-1,0)$ and $p \in [1, 1/\vert \beta \vert)$.\\
The kernel $s\longmapsto g(t-s)$ is singular at $s=t$ and continuously differentiable, strictly increasing and convex on the interval $[0,t)$. Hence, we have the following inequality for $s \in [0,t)$:
 \begin{equation}\label{N-1}
  g(t-s+\e)-g(t-s)\leq\e \sup_{\theta\in(0,1)}|g'(t-s+\theta\e)|\leq\e |g'(t-s)|.
 \end{equation}
 This yields
 \begin{align*}
  |g^\e(t-s)-g(t-s)|^p 
  &\leq \frac{\e^p}{(t-s)^{p\vert \beta \vert}} \Big[ \frac{\vert \beta \vert}{(t-s)} + \lambda \Big]^p \\
 & \leq  2^p \e^p \frac{ \lambda^p }{(t-s)^{p\vert \beta \vert}} +  2^p \e^p\Big[ \frac{\vert \beta \vert}{(t-s)} \Big]^p , 
  \end{align*}
   where we have used the fact that, for $a,b\geq0$ and $p\geq1$:
 \begin{equation}\label{N0}
  (a+b)^p\leq(2\max(a,b))^p\leq(2a)^p+(2b)^p.
 \end{equation}
 Moreover, also the following crude inequality holds for $s \in [0,t)$:
 \begin{align}\label{*}
  \left|(t-s+\e)^\beta e^{-\lambda(t-s+\e)}-(t-s)^\beta e^{-\lambda (t-s)}\right|^p
  &\leq (t-s)^{\beta p}e^{-\lambda (t-s)p}\\
&  \leq (t-s)^{\beta p}.\notag
 \end{align}
 Hence, we obtain the following estimate:
 \begin{align*}
   \int_0^t|g^\e&(t-s)-g(t-s)|^pds\nonumber\\
   &=\int_0^{t-\e}|g^\e(t-s)-g(t-s)|^pds+\int_{t-\e}^t|g^\e(t-s)-g(t-s)|^pds\nonumber\\
   &\leq\e^p\frac{(2\lambda)^p}{1- p\vert\beta \vert}\left[t^{1- p \vert\beta \vert} -\e^{1- p \vert\beta \vert}\right]+\e^p\frac{(2|\beta |)^p}{1- (\vert\beta\vert+1)p}\left[t^{1- (\vert\beta\vert+1)p}-\e^{1- (\vert\beta\vert+1)p}\right]\nonumber\\
   &\quad+\frac{\e^{1-\vert\beta\vert p}}{1-\vert\beta\vert p}\leq\e^{1-\vert\beta\vert p}C(\lambda,\beta,p,t)\longrightarrow 0, \quad \e\to 0,
  \end{align*}
  for the given parameters. 

\vspace{2mm}
 \noindent
 (b)$\quad$Now assume \( \beta>0 \). \\
 The function $g(t-\cdot)$ is zero at \( s=t \) and it is continuously differentiable on \( (0,t) \). For \( s \in (0,t) \), we have that
\begin{align*}
 |g(t-s+\e)-g(t-s)|^p 
 &\leq |e^{-\lambda (t-s)} (e^{-\lambda\e}(t-s+\e)^\beta-(t-s)^\beta)|^p\\
 &\leq 2^p |(t-s+\e)^\beta-(t-s)^\beta|^p + 2^p (t-s)^{\beta p} \vert e^{-\lambda \e}-1 \vert^p\\
 &\leq 2^p \e^p \beta^p \sup_{\theta\in (0,1)} (t-s+\e\theta)^{(\beta-1)p} + 2^p (t-s)^{\beta p} \vert e^{-\lambda \e}-1 \vert^p.
\end{align*}
\noindent
We have to distinguish two cases. If \( \beta\geq 1 \), then we have
\begin{align*}
  \int_0^t|g^\e&(t-s)-g(t-s)|^pds\\
  &\leq (2\e\beta)^p\int_0^t(t-s+\e)^{(\beta-1) p}ds+2 ^p  \vert e^{-\lambda \e}-1 \vert^p  \int_0^t(t-s)^{\beta p}ds\\
  &\leq \frac{(2\e\beta)^p}{1+(\beta-1) p}\big[ (t+\e)^{1+(\beta-1) p} - \e^{ 1+(\beta-1) p }\big] 
+2 ^p  \vert e^{-\lambda \e}-1 \vert^p  \int_0^t(t-s)^{\beta p}ds\\
&\leq \e^p C(\lambda,\beta,p,t) \longrightarrow 0, \quad \e\to 0.
 \end{align*} 
 
\vspace{2mm}
 \noindent
 If \( \beta \in (0, 1) \), then 
 \begin{align*}
  \int_0^t|g^\e&(t-s)-g(t-s)|^pds\\
  &\leq (2\e\beta)^p\int_0^{t-\e}(t-s)^{(\beta-1) p}ds+2 ^p  \vert e^{-\lambda \e}-1 \vert^p  \int_0^{t-\e}(t-s)^{\beta p}ds\\
  &\quad +\int_{t-\e}^t(t-s+\e)^{\beta p}ds\\
  &= \frac{(2\e\beta)^p}{1+(\beta-1) p}\big[ t^{1+(\beta-1) p} - \e^{ 1+(\beta-1) p }\big] 
+\frac{2 ^p  \vert e^{-\lambda \e}-1 \vert^p}{1+\beta p}\big[ t^{1+\beta p} - \e^{ 1+\beta p }\big] \\
&\quad + \frac{\e^{1+\beta p}}{1+\beta p}\\
  &\leq \e^{\min(p,1+\beta p)}C(\lambda,\beta,p,t) \longrightarrow 0, \quad \e\to 0.
   \end{align*}
The estimates are uniform on $t\in [0,T]$ $(T<\infty)$.
\end{example}

\begin{example}\label{Thao}
Here we consider $Y$ to be the Gamma-Volterra process in~\eqref{eq:keyex} with L\'evy driver $L$ associated to the characteristic triplet \( (a,b,\nu) \), with \( b>0 \) and a symmetric measure \( \nu \) such that $\int_\R|z|^p\nu(dz)<\infty$, for some $p$.
In this case Lemma \ref{L-integrands} guarantees that \eqref{eq:keyex} is well defined if, for all \( t\geq0 \), $g(t-\cdot) \in L_p[0,t]$ with \( p\geq2 \).

\noindent This is guaranteed if we have one of the following cases:
\begin{enumerate}
\item[(a)]
$ \lambda \geq 0, \beta \in (-1/2,0)$ and $p \in [2, 1/\vert \beta \vert)$;
\item[(b)]
$\lambda \geq 0, \beta > 0$.
\end{enumerate}
Referring to Theorem \ref{coro:approx1}, the same analysis of Example \ref{ex:Y} leads to the convergence of \( Y(t)^\e \longrightarrow Y(t) \) in \( L_p(\Omega) \) in both cases.
\end{example}
We remark that Example \ref{ex:Y} and Example \ref{Thao} extend in a non-trivial way the work of Thao and Nguyen \cite{thao2}, see Theorem $1$, and also Thao \cite{thao}, see Theorem $2.3$, where an approximation of fractional Brownian motion is considered.


\section{Pathwise Volterra integrals and their approximation}
Now that \( Y \) is well characterized, we proceed by reviewing stochastic integration with respect to $Y$ as integrator.

Naturally, in the case when $Y$ is a semimartingale and the integrand $X$ is predictable, integration can be carried out via It\^o-type calculus with respect to the random measure generated by $Y$.
See e.g. \cite{BJ1983} and \cite{CK2015}.
In \cite{BNBPV} (see also \cite{dinunnovives}) a stochastic integral with respect to $Y$ has been constructed by means of the Malliavin calculus with respect to the Brownian motion and the centered Poisson random measure. This approach does not consider the L\'evy driving noise as a whole, but treats the Gaussian and the centered Poisson random measure separately, and it is well-set when the kernel \( g \) is not degenerate at \( 0 \). 
Also \cite{BM2008} proposes a Skorohod-type integral based on the $S$-transform for a pure jump centered $L$.

In this paper we consider a pathwise-type of integration with respect to $Y$ as introduced in \cite{dinunno} in the lines of \cite{zahle1, zahle2, zahle3}. This is based on fractional calculus, see \cite{SKM1993} for a detailed background.

\subsection{Generalized Lebesgue-Stieltjes integrals with respect to Volterra processes}
First we recall some definitions from fractional calculus, that we are going to use to define the integral of our interest.
\subsubsection*{Elements of fractional calculus}
For a deterministic real-valued function $f\in L_1(a,b)$ $(-\infty < a <b<\infty)$, the Riemann-Liouville left- and right-sided fractional integrals\footnote{In the definitions in \cite{zahle1, zahle2, zahle3} there is a $(-1)^\alpha$ term, originally used by Liouville. The interest in those papers is mostly about harmonic calculus, while in a different context we decided to omit such term.} of order $\alpha>0$ are defined by
\[
 \mathcal{I}^\alpha_{a^+}f(x):=\frac{1}{\Gamma(\alpha)}\int_a^xf(y)(x-y)^{\alpha-1}dy,
\]
and 
\[
 \mathcal{I}^\alpha_{b^-}f(x):=\frac{1}{\Gamma(\alpha)}\int_x^bf(y)(y-x)^{\alpha-1}dy,
\]
respectively, if the integrals converge for a.a. $x \in (a,b)$.
Here $\Gamma$ denotes the Gamma function. The fractional integrals above are well-defined for all $f \in L_q(a,b)$ if $1 \leq q < \frac{1}{\alpha}$. 

\vspace{2mm}
If $f \in L_q(a,b)$ and $g \in L_p(a,b)$ for $ p, q \geq 1 \, : \, \frac{1}{p} + \frac{1}{q} \leq 1+\alpha$ and $\frac{1}{p} + \frac{1}{q} = 1+\alpha$ if $p, q >1$, then the integration by parts
\begin{equation*}
\int_a^b g(x)I_{a^+}^\alpha f(x)dx = \int_a^b f(x)I_{b^-}^\alpha g(x)dx
\end{equation*}
holds. 
This motivates the introduction of the fractional derivatives as a form of inverse operator to the fractional integral. For this we work with a class of functions for which these concepts are well defined. For $q\geq 1$, let $\mathcal{I}^\alpha_{a^+}(L_q)$ be the set of functions $f:(a,b) \longrightarrow \R$ for which there exists $\varphi 
\in L_q(a,b)$ such that $f=\mathcal{I}^\alpha_{a^+}\varphi$. It can be shown that the function $\varphi$ is unique in $L_q(a,b)$ (see \cite{Mishura-book} Lemma 1.1.2  and comments). Also, if $q>1$, $f \in I_{a^+}^\alpha(L_q)$ if and only if $ f \in L_q(a,b)$ and there is $L_q$-convergence for $\delta \downarrow 0$ of the function
\begin{equation*}
\int_a^{x-\delta} \frac{f(x)-f(y)}{(x-y)^{\alpha+1}}dy, \qquad x \in (a,b),
\end{equation*}
(where $f(y)=0$ for $y \notin [a,b]$). The conditions are sufficient if $q=1$. 
Analogously, set $\mathcal{I}^\alpha_{b^-}(L_q)$ to be the set of functions $f$ for which there exists $\varphi \in L_q(a,b)$ such that $f=\mathcal{I}^\alpha_{b^-}\varphi$.
For $q>1$ we have that  $f \in I_{b^-}^\alpha(L_q)$ if and only if $ f \in L_q(a,b)$ and there is $L_q$-convergence for $\delta \downarrow 0$ of the function
\begin{equation*}
\int_{x+\delta}^b \frac{f(x)-f(y)}{(y-x)^{\alpha+1}}dy, \qquad x \in (a,b).
\end{equation*}
Again the conditions are sufficient if $q=1$. 

Furthermore, for $\alpha\in (0,1)$ and all $f\in \mathcal{I}^\alpha_{a^+}(L_q)$, the function $\varphi$ coincides $a.e.$ with the Riemann-Liouville left-sided fractional derivative defined as the inverse operator of $\mathcal{I}^\alpha_{a^+}$. Namely, $\varphi$ is $a.e.$ equal to
\begin{equation*}
\mathcal{D}^\alpha_{a^+}f(x)=\frac{d}{dx} \mathcal{I}^{1-\alpha}_{a^+} f(x), \quad x\in (a,b).
\end{equation*}
Correspondingly, for $f \in I_{b^-}^\alpha(L_q)$, we have the right-sided fractional derivative
\begin{equation*}
\mathcal{D}^\alpha_{b^-}f(x)= - \frac{d}{dx} \mathcal{I}^{1-\alpha}_{b^-} f(x), \quad x\in (a,b).
\end{equation*}
In this cases the Riemann-Liouville fractional derivatives admit the respective Weyl representations:
\begin{align*}
 \mathcal{D}^\alpha_{a^+}f(x)&=\frac{1}{\Gamma(1-\alpha)}\left(\frac{f(x)}{(x-a)^\alpha}+\alpha\int_a^x\frac{f(x)-f(y)}{(x-y)^{\alpha+1}}dy\right)\mathbf{1}_{(a,b)}(x),\\
 \mathcal{D}^\alpha_{b^-}f(x)&=\frac{1}{\Gamma(1-\alpha)}\left(\frac{f(x)}{(b-x)^\alpha}+\alpha\int_x^b\frac{f(x)-f(y)}{(y-x)^{\alpha+1}}dy\right)\mathbf{1}_{(a,b)}(x).
\end{align*}
The convergence of the integrals is in $L_q$, if $q>1$, and it is pointwise $a.e.$, if $q=1$. We recall that, for all $\alpha\in (0,1)$, for all $f\in \mathcal{C}^1(a,b)$, the derivatives $\mathcal{D}^\alpha_{a^+}f$ and $\mathcal{D}^\alpha_{b^-}f$ exist and are in $L_q(a,b)$ for $1 \leq q < \frac{1}{\alpha}$.

\vspace{2mm}
Let $f,g:[a,b]\rightarrow\R$. Assume that the limits
\[
 f(t^+):=\lim_{\delta\searrow0}f(t+\delta),\qquad g(t^-):=\lim_{\delta\searrow0}g(t-\delta),
\]
exist for $a\leq t\leq b$ and denote
\begin{align*}
 f_{a^+}(x)&:=\mathbf{1}_{(a,b)}(x)\bigl(f(x)-f(a^+)\bigr),\\
 g_{b^-}(x)&:=\mathbf{1}_{(a,b)}(x)\bigl(g(b^-)-g(x)\bigr).
\end{align*}
\begin{definition}\label{detGLS}
Assume that $f_{a^+}\in\mathcal{I}^\alpha_{a^+}(L_q)$ and $g_{b^-}\in\mathcal{I}^{1-\alpha}_{b^-}(L_p)$ for some $p^{-1}+q^{-1}\leq1$, and $0<\alpha<1$. The \textit{generalized fractional Lebesgue-Stieltjes integral} of $f$ with respect to $g$ is defined by
 \begin{equation*}
  \int_a^bf(x)dg(x):=\int_a^b\mathcal{D}^\alpha_{a^+}f_{a^+}(x)\mathcal{D}^{1-\alpha}_{b^-}g_{b^-}(x)dx+f(a^+)(g(b^-)-g(a^+)).
 \end{equation*}
\end{definition}
Naturally, the conditions $f_{a^+}\in\mathcal{I}^\alpha_{a^+}(L_q)$ and $g_{b^-}\in\mathcal{I}^{1-\alpha}_{b^-}(L_p)$ mean that $\mathcal{D}^\alpha_{a^+}f_{a^+} \in L_q(a,b)$ and $ \mathcal{D}^{1-\alpha}_{b^-}g_{b^-}\in L_p(a,b)$.
Hence the integral on the right-hand side is well-defined. It can be shown that the definition of the integral does not depend on $\alpha$, see \cite[Proposition 2.1]{zahle1}.
 
Moreover, for $1 \leq q < \frac{1}{\alpha}$, we have that  $f_{a^+}\in\mathcal{I}^\alpha_{a^+}(L_q)$ if and only if $f\in\mathcal{I}^\alpha_{a^+}(L_q)$ and $f(a^+)$ exists. Then the generalized Lebesgue-Stieltjes integral admits a simplified representation as
 \begin{align*}
  \int_a^bf(x)dg(x):=\int_a^b\mathcal{D}^\alpha_{a^+}f(x)\mathcal{D}^{1-\alpha}_{b^-}g_{b^-}(x)dx.
 \end{align*}

\vspace{2mm}
\noindent
Motivated by the above considerations, the following definition can be given, see \cite{dinunno}.
\begin{definition}\label{stoGLS}
Two real-valued measurable stochastic processes $X=X(t)$, $t\geq0$, and 
$Y=Y(t)$, $t\geq0$, are \emph{fractionally $\alpha$-connected for some $t$ and for some $\alpha\in (0,1)$}, if the generalized Lebesgue-Stieltjes integral
\begin{equation}\label{eq:gLSint}
 \int_0^tX(s)dY(s):=\int_0^t\left(\mathcal{D}^\alpha_{0^+}X\right)(s)\left(\mathcal{D}^{1-\alpha}_{t^-}Y_{t^-}\right)(s)ds,
\end{equation}
exists \( \prob \)-$a.s.$
\end{definition}
From fractional calculus we recall that the integral above exists and does not depend on $\alpha$ whenever $X\in \mathcal{I}^\alpha_{0+}(L_q)$ and $Y_{t^-} \in \mathcal{I}^{1-\alpha}_{t^-}(L_p)$ $\prob$-$a.s.$ Here the random variables $Y(t^-)$ are well-defined, see Lemma \ref{lemma:leftlimit}.
In general, we know that the integral above exists if $\mathcal{D}^\alpha_{0^+}X \in L_q(0,t)$ and $ \mathcal{D}^{1-\alpha}_{t^-}Y_{t^-}\in L_p(0,t)$, $\prob$-$a.s.$ for some $p^{-1} + q^{-1} =1$. Then the following definitions appear naturally.

On the time horizon $[0,T]$ $(T<\infty)$, for $p,q\in[1,\infty) \, : \, p^{-1}+q^{-1}=1$ and $0<\alpha<1$, define the sets of stochastic integrands \( X=X(t),\;t\in[0,T] \):
\begin{align*}
 \mathcal{D}^+_{q}(\alpha,T)&:=\Big{\{}X:\;\int_0^T|\left(\mathcal{D}^\alpha_{0^+}X\right)(s)|^qds<\infty\;a.s.\Big{\}},\\
 \mathcal{D}^+_{\infty}(\alpha,T)&:=\Big{\{}X:\;\sup_{0\leq s\leq T}|\left(\mathcal{D}^\alpha_{0^+}X\right)(s)|<\infty\;a.s.\Big{\}},
\end{align*}
and integrators $Y=Y(t),\;t\in[0,T]$:
\begin{align*}
 \mathcal{D}^-_{p}(\alpha,T)&:=\Big{\{}Y:\;\int_0^t|\left(\mathcal{D}^{1-\alpha}_{t^-}Y_{t^-}\right)(s)|^pds<\infty\;a.s.,\;t\in[0,T]\Big{\}},\\
 \mathcal{D}^-_{\infty}(\alpha,T)&:=\Big{\{}Y:\;\sup_{0\leq s\leq t}|\left(\mathcal{D}^{1-\alpha}_{t^-}Y_{t^-}\right)(s)|<\infty\;a.s.,\;t\in[0,T]\Big{\}}.
\end{align*}
It is easy to see that the couples $(X,Y)\in\mathcal{D}^+_{1}(\alpha,T)\times\mathcal{D}^-_{\infty}(\alpha,T)$, $(X,Y)\in\mathcal{D}^+_{\infty}(\alpha,T)\times\mathcal{D}^-_{1}(\alpha,T)$, and $(X,Y)\in\mathcal{D}^+_{q}(\alpha,T)\times\mathcal{D}^-_{p}(\alpha,T)$ are fractionally $\alpha$-connected for all $t \in [0,T]$.
Then we say that the elements in $\mathcal{D}^-_{p}(\alpha,T)$ are the \emph{appropriate $(p,\alpha)-$integrators}, $p\in[1,\infty]$, for the elements in $\mathcal{D}^+_{q}(\alpha,T)$, $q=\frac{p}{p-1}$, with the conventions that $\frac{1}{0}=\infty$ and $\frac{\infty}{\infty}=1$.

\vspace{2mm}
Hereafter we formulate the concept of two processes being fractionally $\alpha$-connected in terms of expectations. This is a direct consequence of Theorem \ref{thm:dinunno}. We define new classes of integrands and integrator processes with conditions that are easier to verify and which are included in the previously given classes.
We define the sets:
\begin{align*}
 \E\mathcal{D}^-_{p}(\alpha,T)&:=\Big{\{}Y:\;\int_0^t\E|\left(\mathcal{D}^{1-\alpha}_{t^-}Y_{t^-}\right)(s)|^pds<\infty\;,\;t\in[0,T]\Big{\}}\subset\mathcal{D}^-_{p}(\alpha,T),\\
 \E\mathcal{D}^-_{\infty}(\alpha,T)&:=\Big{\{}Y:\;\sup_{0\leq s\leq t}\E|\left(\mathcal{D}^{1-\alpha}_{t^-}Y_{t^-}\right)(s)|<\infty\;,\;t\in[0,T]\Big{\}}\subset\mathcal{D}^-_{\infty}(\alpha,T),
\end{align*}
and
\begin{align*}
 \E\mathcal{D}^+_{q}(\alpha,T)&:=\Big{\{}X:\;\int_0^T\E|\left(\mathcal{D}^\alpha_{0^+}X\right)(s)|^qds<\infty\Big{\}}\subset\mathcal{D}^+_{q}(\alpha,T),\\
 \E\mathcal{D}^+_{\infty}(\alpha,T)&:=\Big{\{}X:\;\sup_{0\leq s\leq T}\E|\left(\mathcal{D}^\alpha_{0^+}X\right)(s)|<\infty\Big{\}}\subset\mathcal{D}^+_{\infty}(\alpha,T).
\end{align*}
Again, we say that the elements in $\E\mathcal{D}^-_{p}(\alpha,T)$ are the \emph{appropriate\\ $(p,\alpha)-$integrators}, $p\in[1,\infty]$, for the elements in $\E\mathcal{D}^+_{q}(\alpha,T)$, $q=\frac{p}{p-1}$.

\vspace{3mm}
\noindent\textbf{Remark.} It is easy to see that for the couples $(X,Y) \in \E\mathcal{D}^+_{q}(\alpha,T)\times \E\mathcal{D}^-_{p}(\alpha,T)$, the generalized Lebesgue-Stieltjes integral \eqref{eq:gLSint} exists both $\mathbb{P}$-$a.s.$ and in $L_1(\mathbb{P})$.

\vspace{3mm}
\noindent\textbf{Remark: Relationship with other types of stochastic integration.}
The definition of generalised Lebesgue-Stieltjes integral (Definition \ref{detGLS}) can be extended, see \cite{zahle2} Definition 4.4, which is motivated by Lemma 4.1 and Lemma 4.2 in the same reference. This leads to the Definition 4.7 of stochastic integral in \cite{zahle2}, which also extends the forward integral introduced in \cite{russo} Section 1.

{\it All these stochastic integrals coincide with the It\^o integral whenever the integrator is a semimartingale, the integrand is an adapted c\`agl\`ad process, and the convergences are uniformly on compacts in probability (ucp).}
See \cite{russo} Proposition 1.1, \cite{zahle2} Proposition 4.9, \cite{zahle3} Section 5.
This result also applies to the stochastic integral of Definition \ref{stoGLS} in the present paper. In fact two functionally $\alpha$-connected processes are integrable as per \cite{zahle2} Definition 4.7.

\subsubsection*{A L\'evy driven Volterra process as integrator}
\noindent We now review the case of L\'evy driven Volterra processes \eqref{eq:mainint} as integrators. The following result relies on the estimate of Theorem \ref{thm:dinunno} (b). See \cite{dinunno} Section $5$.

\begin{theorem} \label{thm:dinunno1}
Let \( Y=Y(t)=\int_0^tg(t-s)dL(s),\;t\in[0,T], \) be a Volterra process where \( L=L(t),\;t\geq0 \) is a L\'evy process with the characteristic triplet \( (a,b,\nu) \), for \( b\geq0 \), and \( \nu \) a symmetric L\'evy measure such that $\int_\R|z|^p\nu(dz)<\infty$ for some $p$ with $p\geq1$ if $b=0$ or $p\geq2$ if $b>0$. Moreover, for this value of \( p \), assume that $g=g(t-\cdot)\in L_p[0,t]$ for any $t\in[0,T]$ and the following set of conditions for some \( \alpha\in(0,1) \):

 \noindent\textbf{Assumptions ($D_p$)}
 \begin{enumerate}[(i)]
  \item $\int_0^t(t-s)^{\alpha p-p}\left(\int_s^t|g(t-v)|^pdv\right)ds<\infty$,
  \item $\int_0^t(t-s)^{\alpha p-p}\left(\int_0^s|g(t-v)-g(s-v)|^pdv\right)ds<\infty$,
  \item $\int_0^t\int_s^t(u-s)^{\alpha p-2p}\left(\int_s^u|g(u-v)|^pdv\right)duds<\infty$,
  \item $\int_0^t\int_s^t(u-s)^{\alpha p-2p}\left(\int_0^s|g(u-v)-g(s-v)|^pdv\right)duds<\infty$.
 \end{enumerate}
 Then $Y\in\E\mathcal{D}^-_p(\alpha,T)$, so, $Y$ is an appropriate $(p,\alpha)$-integrator for any $X\in\mathcal{D}^+_q(\alpha,T)$ with $q^{-1}+p^{-1}=1$.
\end{theorem}

\vspace{2mm}
\begin{example}\label{ex:appr}
 \textbf{The Gamma-Volterra process \( Y \) as an appropriate $(p,\alpha)$-integrator.} 
 In this example, we find the conditions on the parameters $\alpha, p$ depending on $ \beta, \lambda $ so that the Gamma-Volterra process $Y$ in \eqref{eq:keyex} is an appropriate $(p,\alpha)$-integrator. 
 From Lemma \ref{L-integrands} and Example~\ref{ex:Y} we already know that $Y$ is well defined if $\beta>0$ and if $\beta \in (-1,0)$ with $ 1 + \beta p > 0 $. 
 We consider a L\'evy driving noise with characteristic triplet $(a,0,\nu)$.
 The case $(a,b,\nu)$ with $b>0$ is treated similarly, cf. Example \ref{Thao}.
 
 Recall the set of conditions on the kernel function \( g \) in Theorem \ref{thm:dinunno1}. We go through the list, and find conditions on the parameters $\alpha, \beta$ and $p$ in order for (i)-(iv) of ($D_p$) to be satisfied.
 \begin{enumerate}[(i)]
  \item The innermost integral in (i) can be estimated by the following:
   \begin{align*}
    \int_s^t\left|(t-v)^\beta e^{-\lambda(t-v)}\right|^pdv&\leq \int_s^t(t-v)^{\beta p}dv=\frac{(t-s)^{1+\beta p}}{1+\beta p},
   \end{align*}
   where the integral is well defined since $1+\beta p>0$. We calculate the outer integral of (i), and find an estimate:
   \begin{align*}
    \int_0^t(t-s)^{\alpha p-p}\left(\int_s^t|g(t-v)|^pdv\right)ds&\leq\int_0^t(t-s)^{\alpha p-p}\frac{(t-s)^{1+\beta p}}{1+\beta p}ds\\
    &=\frac{1}{1+\beta p}\int_0^t(t-s)^{1+(\alpha+\beta-1)p}ds\\
    &=\frac{t^{2+(\alpha+\beta-1)p}}{(1+\beta p)(2+(\alpha+\beta-1)p)},
   \end{align*}
   where the integral is well defined when $2+(\alpha+\beta-1)p>0$.
  \item We need to separate the cases in which $\beta$ is positive or negative.\\
  For $\beta > 0$ , by \eqref{N0}, we have that
  \begin{align}
  \left|(t-v)^\beta e^{-\lambda(t-v)}-(s-v)^\beta e^{-\lambda (s-v)}\right|^p &\leq 2^p(t-v)^{\beta p}e^{-\lambda (s-v)p} \label{N1}\\
  &\leq 2^p(t-v)^{\beta p} \nonumber.
  \end{align}
  Hence we can estimate the integral in (ii) as follows:
  \begin{align*}
   \int_0^t(t-s)^{\alpha p-p}&\left(\int_0^s|g(t-v)-g(s-v)|^pdv\right)ds\\
   & \leq \int_0^t 2^p (t-s)^{\alpha p-p}\left(\int_0^s (t-v)^{\beta p} dv\right)ds\\
   & = \int_0^t \frac{2^p}{1 + \beta p} \Bigl( (t-s)^{\alpha p - p + 1 + \beta p} - (t-s)^{\alpha p - p} \; t^{1 + \beta p}\Bigr)ds.
  \end{align*}
  The integral above is finite for $ 1 + \alpha p - p > 0 $.\\
  For $\beta < 0$ we have that
  \begin{equation}\label{N2}
  \left|(t-v)^\beta e^{-\lambda(t-v)}-(s-v)^\beta e^{-\lambda (s-v)}\right|^p \leq (s-v)^{\beta p}.
  \end{equation}
  Then we have the following estimate for the integral in (ii):
  \begin{align*}
   \int_0^t(t-s)^{\alpha p-p}&\left(\int_0^s|g(t-v)-g(s-v)|^pdv\right)ds\\
   & \leq \int_0^t (t-s)^{\alpha p-p}\left(\int_0^s (s-v)^{\beta p} dv\right)ds.
  \end{align*}
  The innermost integral is finite as $1 + \beta p>0$ and increasing in $s$. Then the estimate above is finite for $1+\alpha p - p > 0$.
  \item The innermost integrals of (i) and (iii) are the same with the only attention to be given to the range of integration, so:
  \[
   \int_s^u\left|(u-v)^\beta e^{-\lambda(u-v)}\right|^pdv\leq\frac{(u-s)^{1+\beta p}}{1+\beta p},
  \]
  which is well-defined as $1+\beta p>0$. The second layer of integrals is then dominated by
  \begin{equation*}
   \int_s^t(u-s)^{\alpha p-2p}\left(\frac{(u-s)^{1+\beta p}}{1+\beta p}\right)du=\frac{1}{1+\beta p}\int_s^t(u-s)^{1+p(\alpha+\beta-2)}du,
  \end{equation*}
  which is finite whenever $2+(\alpha+\beta-2)p>0$. The outermost integral of (iii) is then also clearly finite.
  \item Similar to the study of (ii) here we also have to separate the cases $\beta \in (-1,0)$ and $\beta >0$. Moreover, our estimates need to be sharper than before.
  
  For $\beta<0$, we go through the integral:
  \begin{align}
   \int_0^t&\int_s^t(u-s)^{\alpha p-2p}\left(\int_0^s|g(u-v)-g(s-v)|^pdv\right)duds \label{N3}\\
   &= A_1 + A_2 + B \nonumber
   \end{align}
   by splitting the integration range in an opportune way. We start by considering
   \begin{align*}
   A_1:=&\int_0^{t/2}\int_s^{2s}(u-s)^{\alpha p-2p}\left(\int_0^{2s-u}|g(u-v)-g(s-v)|^pdv \right.\\
   & \hspace{4cm} \left. +\int_{2s-u}^s|g(u-v)-g(s-v)|^pdv\right)duds.
   \end{align*}
   By application of \eqref{N-1} with $(u-s)$ in the place of $\e$, we observe that
   \begin{equation*}
   |g(u-v)-g(s-v)|^p \leq (u-s)^p |g'(s-v)|^p \leq (u-s)^p \beta^p (s-v)^{(\beta-1)p}.
   \end{equation*}
   Thus we have
   \begin{align*}
   A_1 &\leq \int_0^{t/2}\int_s^{2s}(u-s)^{\alpha p-2p}\left(\int_0^{2s-u}|\beta|^p (u-s)^p (s-v)^{(\beta-1)p}dv \right.\\
   & \hspace{4.5cm} \left. +\int_{2s-u}^s (s-v)^{\beta p}dv\right)duds\\
   &= \int_0^{t/2}\int_s^{2s} \left(\frac{|\beta|^p (u-s)^{\alpha p -p}}{1+(\beta-1)p} (s^{1+(\beta-1)p} - (u-s)^{1+(\beta-1)p}) \right.\\
   & \hspace{2.5cm} \left. +\frac{(u-s)^{1+\beta p +\alpha p - 2p}}{1+\beta p} \right)duds\\
   &= \int_0^{t/2} C(\beta,\alpha,p) \; s^{2+\alpha p +\beta p - 2p} ds,
   \end{align*}
   This integral is finite when $2 + (\alpha + \beta - 2)p > 0$. Then, using \eqref{N2}, we consider
   \begin{align*}
   A_2 &:=\int_0^{t/2}\int_{2s}^t(u-s)^{\alpha p-2p}\left(\int_0^s|g(u-v)-g(s-v)|^pdv\right)duds\\
   & \leq \int_0^{t/2}\Bigl(\int_{2s}^t(u-s)^{\alpha p-2p}du\Bigr)\Bigl(\int_0^s (s-v)^{\beta p} dv\Bigr)ds\\
   &= \int_0^{t/2}\frac{s^{1+\beta p}}{1+\beta p}\Bigl(\int_{2s}^t(u-s)^{\alpha p-2p}du\Bigr)ds\\
   &= \int_0^{t/2} \frac{s^{1+\beta p}}{(1+\beta p)(1+\alpha p - 2p)}((t-s)^{1+\alpha p-2p} - s^{1+\alpha p - 2p})ds,
   \end{align*}
   which is finite for $2 + \alpha p - 2p > 0$. The last summand in \eqref{N3} is given by
   \begin{equation*}
   B:=\int_{t/2}^t\int_s^t(u-s)^{\alpha p-2p}\left(\int_0^s|g(u-v)-g(s-v)|^pdv\right)duds.
   \end{equation*}
  By using the same estimates as for $A_2$, we get that $B$ is finite for $2 + \alpha p - 2p > 0$.
  
  For $\beta >0$, as in case (b) in Example \ref{ex:Y}, we obtain the following inequality:
  \begin{align*}
  |g(u-v)-g(s-v)|^p &\leq 2^p (u-s)^p \beta^p \sup_{\theta\in (0,1)} (s-v+\theta(u-s))^{(\beta-1)p}\\
  & \quad + 2^p (s-v)^{\beta p} \vert e^{-\lambda(u-s)}-1 \vert^p,
  \end{align*}
  and again we have to distinguish two cases. If $\beta \geq 1$, then we have
  \begin{align*}
  \int_0^t&\int_s^t(u-s)^{\alpha p-2p}\left(\int_0^s|g(u-v)-g(s-v)|^pdv\right)duds\\
  & \leq \int_0^t\int_s^t(u-s)^{\alpha p-2p} \Bigl(\frac{2^p (u-s)^p \beta^p}{1+(\beta -1)p}(u^{1+(\beta -1)p} - (u-s)^{1+(\beta -1)p}) \Bigr.\\
  & \hspace{4cm} \Bigl. + \frac{2^p \lambda^p(u-s)^p}{1+\beta p}s^{1+\beta p} \Bigr)duds\\
  & \leq C(\lambda, \beta, p, t) \int_0^t\int_s^t(u-s)^{\alpha p-p}duds.
  \end{align*}
  Then the integral is finite for $1+\alpha p - p > 0$. If $0 < \beta < 1$, we have
  \begin{align*}
  \int_0^t&\int_s^t(u-s)^{\alpha p-2p}\left(\int_0^s|g(u-v)-g(s-v)|^pdv\right)duds\\
  & \leq \int_0^t\int_s^t(u-s)^{\alpha p-2p} \Bigl(\frac{2^p (u-s)^p \beta^p}{1+(\beta -1)p}s^{1+(\beta -1)p} \Bigr.\\
  & \hspace{4cm} \Bigl. + \frac{2^p \lambda^p(u-s)^p}{1+\beta p}s^{1+\beta p} \Bigr)duds\\
  & \leq C(\lambda, \beta, p, t) \int_0^t\int_s^t(u-s)^{\alpha p-p}duds,
  \end{align*}
  which is finite for $1+\alpha p - p > 0$ and $1+(\beta - 1)p >0$.
 \end{enumerate}
 Summarising, the following conditions on the parameters are sufficient for the Gamma-Volterra process \eqref{eq:keyex} to be a $(p,\alpha)$-integrator:
 \begin{itemize}
 \item
 $\beta \geq 1, 1+(\alpha-1)p>0$
 \item
 $\beta\in (0,1), 1+(\alpha-1)p>0, 1 + (\beta-1)p>0$
 \item
 $\beta \in (-1,0), 2+(\alpha+\beta-2)p >0$.
 \end{itemize}
\end{example}

\subsection{Approximation of integrals with Volterra drivers}

We are now ready to study the approximation of integrals with respect to Volterra processes by integrals driven by semimartingales. Further on we will consider again the example of the Gamma-Volterra process in \eqref{eq:keyex}.

\begin{theorem}\label{prop:approx2}
Let \( L=L(t),\;t\geq0 \) be a L\'evy process with the characteristic triplet \( (a,b,\nu) \), for \( a\in\R,\;b\geq0 \), and \( \nu \) a symmetric measure such that $\int_\R|z|^p\nu(dz)<\infty$ for some $p$ with $p\geq1$ if $b=0$ or $p\geq2$ if $b>0$. Let the kernel functions $g=g(t-\cdot)$ and $g^\e=g^\e(t-\cdot)$ belong to $L_p[0,t]$ for any $t\in[0,T]$.
Assume that $g(t^--\cdot)$ and $g^\e(t^--\cdot)$ are well-defined and in $L_p([0,t])$ and assume that $g, g^\e$ also satisfy the set of conditions \( (D_p) \) for some \( \alpha\in(0,1) \) together with the following:

 \noindent\textbf{Assumptions ($C_p$)}. For $\e\to 0$,
 \begin{align*}
  (i) &\int_0^T\int_s^T\frac{|g^\e(T^--v)-g(T^--v)|^p}{(T-s)^{p-\alpha p}}dvds\rightarrow0\\
  (ii) &\int_0^T\int_0^s\frac{|(g^\e(T^--v)-g(T^- -v))-(g^\e(s-v)-g(s-v))|^p}{(T-s)^{p-\alpha p}}dvds\rightarrow0\\
  (iii) &\int_0^T\int_s^T\int_s^u\frac{|g^\e(u-v)-g(u-v)|^p}{(u-s)^{2p-\alpha p}}dvduds\rightarrow0\\
  (iv) &\int_0^T\int_s^T\int_0^s\frac{|(g^\e(u-v)-g(u-v))-(g^\e(s-v)-g(s-v))|^p}{(u-s)^{2p-\alpha p}}dvduds\rightarrow0
 \end{align*}
Define the Volterra processes
 \[
 Y:=Y(s)=\int_0^tg(s-v)dL(v),\qquad s\in[0,T],
 \]
 and
 \[
 Y^\e:=Y^\e(s)=\int_0^tg^\e(s-v)dL(v),\quad s\in[0,T].
 \]
 Then, for any stochastic process $X\in\E\mathcal{D}^+_{q}(\alpha,T)$ where $p^{-1}+q^{-1}=1$, we have the convergence
 \[
  \int_0^T X(s)dY^\e(s)\longrightarrow\int_0^T X(s)dY(s),\quad\text{ as }\e\rightarrow0,
 \]
  in \( L_1(\Omega) \) of the generalised Lebesgue-Stieltjes integrals.
\end{theorem}
\begin{proof}
In the given setting for the processes $Y,Y^\e$ and $X$, the generalised Lebesgue-Stieltjes integrals are well-defined $\prob$-$a.s.$ and in $L_1(\Omega)$. See also Lemma \ref{lemma:leftlimit} for the definition of $g(t^--\cdot)$,$g^{\epsilon}(t^--\cdot)$ and $Y(t^-)$,$Y^{\epsilon}(t^-)$.
By linearity of the operators involved and the use of the H\"older inequality, we estimate the $L_1(\Omega)$ difference of the integrals as follows:
 \begin{align}\label{eq:L1dist}
  &\E\Bigg|\int_0^T X(s)dY^\e(s)-\int_0^T X(s)dY(s)\Bigg|\nonumber\\
  &\leq\int_0^T\|(\mathcal{D}^\alpha_{0^+}X)(s)\|_{L_q(\Omega)}\|(\mathcal{D}^{1-\alpha}_{T^-}(Y^\e-Y)_{T^-})(s)\|_{L_p(\Omega)}ds \\
  &\leq\left(\int_0^T\|(\mathcal{D}^\alpha_{0^+}X)(s)\|^q_{L_q(\Omega)}ds\right)^\frac{1}{q}
  \left(\int_0^T\|(\mathcal{D}^{1-\alpha}_{T^-}(Y^\e-Y)_{T^-})(s)\|^p_{L_p(\Omega)}ds\right)^\frac{1}{p}.\nonumber
 \end{align}
 Hence the statement is proved if
 \begin{equation}
 \label{eq:*1}
 \int_0^T \E \big[\big\vert (\mathcal{D}^{1-\alpha}_{T^-}(Y^\e-Y)_{T^-})(s)\big\vert^p \big]ds \longrightarrow 0, \quad \e \downarrow 0.
 \end{equation}
Define $\bar g^\e:= g^\e - g$ and 
$$
\bar Y^\e := Y^\e - Y = \int_0^T \bar g^\e (s-v) dL(v), \quad s\geq 0.
$$ 
From $\bar Y^\e_{T^-} (s) := \big( \bar Y^\e (T^-) - \bar Y^\e (s) \big) 1_{(0,T)}(s)$,
we can see that, if $\bar Y^\e_{T^-} \in \mathcal{I}^{1-\alpha}_{T^-} (L_p)$, the fractional derivative is well-defined and admits representation
 \begin{align}
 \label{eq:*2}
 (\mathcal{D}^{1-\alpha}_{T^-}\bar Y^\e_{T^-}) (s) 
 =& 
 \frac{(-1)^{1-\alpha}}{\Gamma(\alpha)}
 \Big[\frac{\bar Y^\e(T^-) - \bar Y^\e(s)}{(T-s)^{1-\alpha}}\\
 &+ (1-\alpha) \int_s^T\frac{\bar Y^\e(s)-\bar Y^\e(u)}{(u-s)^{2-\alpha}}du \Big]\mathbf{1}_{(0,T)}(s), \nonumber
  \end{align}
 which can then be substituted into \eqref{eq:*1}.
 Observe that, for $y\leq x$, we have
 \begin{align*}
 \bar Y^\e(x) - \bar Y^\e(y)
=& \int_y^x \bar g^\e (x-v) dL(v)
+ \int_0^y [ \bar g^\e (x-v) - \bar g^\e (y-v)] dL(v).
 \end{align*}
 Hence, from the moment estimates of Theorem \ref{thm:dinunno}, we obtain
 \begin{align}
 \label{eq:*3}
 \int_0^T \int_s^T & \frac{\E[\vert  \bar Y^\e(s) - \bar Y^\e(u) \vert^p] }{(u-s)^{(2-\alpha)p}} \,duds 
 \leq C_1 \Big[ \int_0^T \int_s^T \frac{\Vert \bar g^\e (u-\cdot) \Vert_{L_p(s,u]}^p}{(u-s)^{(2-\alpha)p} } \, duds \nonumber \\
 &+  \int_0^T \int_s^T \frac{ \Vert \bar g^\e (u-\cdot)- \bar g^\e (s-\cdot) \Vert_{L_p(0,s]}^p}{(u-s)^{(2-\alpha)p} } \, duds \Big]
 \end{align}
 which vanishes as $\e \downarrow 0$, thanks to $(D_p)(iii)$-$(iv)$ and $(C_p)(iii)$-$(iv)$. 
 Similarly, we can see that $(D_p)(i)$-$(ii)$ and $(C_p)(i)$-$(ii)$ ensure the convergence 
  \begin{align}
 \label{eq:*4}
 \int_0^T  & \frac{\E[\vert  \bar Y^\e(T^-) - \bar Y^\e(s) \vert^p] }{(T-s)^{(1-\alpha)p}} \,ds 
 \leq C_2 \Big[ \int_0^T \frac{\Vert \bar g^\e (T^--\cdot) \Vert_{L_p(s,T]}^p}{(T-s)^{(1-\alpha)p} } \, ds \nonumber \\
 &+  \int_0^T \frac{ \Vert \bar g^\e (T^--\cdot)- \bar g^\e (s-\cdot) \Vert_{L_p(0,s]}^p}{(T-s)^{(1-\alpha)p} } \, ds \Big]
 \longrightarrow 0, \quad \e \downarrow 0.
 \end{align}
 Naturally, \eqref{eq:*3}-\eqref{eq:*4} guarantee \eqref{eq:*1}.
 To conclude the proof, we verify that $\bar Y^\e_{T^-} \in \mathcal{I}^{1-\alpha}_{T^-}(L_p)$.
 For $p \geq 1$, this is ensured when $\bar Y^\e_{T^-} \in L_p([0,T])$ $\prob$-$a.s.$ and the processes
 $$
 A_\delta(s) := \int_{s+\delta}^T \frac{\bar Y^\e (s) - \bar Y^\e(u)}{(u-s)^{(2-\alpha)}} du, \quad s \in [0,T],
 $$
 converges in $L_p([0,T])$, $\prob$-$a.s.$ for $\delta\downarrow 0$. We verify these two requirements. 
First we see that we have that $\bar Y^\e_{T^-} \in L_p(\Omega\times [0,T])$ from the estimates of Theorem \ref{thm:dinunno}.
By dominated convergence we can see that $A_\delta$ converges to $A_0$ in $L_p([0,T])$.In fact, for all $\delta$ small, we have 
\[
 \begin{split}
 \vert A_\delta(s) - A_0(s) \vert^p 
 &= \Big\vert \int_s^{s+\delta} \frac{\bar Y^\e (s) - \bar Y^\e(u)}{(u-s)^{(2-\alpha)}} du \Big\vert^p\\
  &\leq \int_s^T \frac{|\bar Y^\e (s) - \bar Y^\e(u)|^p}{(u-s)^{(2-\alpha)p}} du =: B(s)
 \end{split}
 \]
$\prob$-$a.s.$. and the bound $B \in L_p(\Omega\times [0,T])$ (see \eqref{eq:*3}).
By this the proof is complete. 
\end{proof}

\vspace{2mm}
\begin{example}\label{ex:general}
 \textbf{Approximation of the integrals with respect to Gamma-Volterra processes.}
 For illustration, we consider the process \eqref{eq:keyex} with $\lambda=0$, in the kernel function, that is 
 $g(t-s)=(t-s)^\beta$. To treat the case $\lambda >0$, we need similar inequalities as in Example \ref{ex:appr}.

As before, we consider a L\'evy driving noise with characteristic triplet $(0,0,\nu)$ with $\nu$ symmetric and such that
$\int_\mathbb{R} \vert z \vert^p \nu(dz) <\infty $ for some $p$.
The parameters $\beta, p , \alpha$ satisfy the conditions of Example \ref{ex:Y} and Example \ref{ex:appr}.
These guarantee the validity of assumptions $(D_p)$.
 We now go through the requirements of $(C_p)$ with $g^\e(t-s) := (t-s+\e)^\beta$, $s\in [0,t]$.
 \begin{enumerate}[(i)]
  \item 
  Using the same approach and inequalities as in Example~\ref{ex:Y}, we split the inner integral. Again we separate the cases depending on $\beta$. Taking $\beta<0$ we obtain
  \begin{align*}
   \int_s^T|g^\e(T^--v)-g(T^-- & v)|^pdv\leq \int_s^{T-\e}\e^p|\beta|^p(T-v)^{(\beta -1)p}dv\\
   &\hspace{1cm} + \int_{T-\e}^T (T-v)^{\beta p}dv\\
   &=\frac{|\beta|^p\e^p}{1+(\beta-1)p}((T-s)^{1+(\beta-1)p} - \e^{1+(\beta-1)p})\\
   &\hspace{1cm} + \frac{\e^{1+\beta p}}{1+\beta p}.
  \end{align*}
  The next layer of the integrals yields:
  \begin{align*}
   \int_0^T(T-s)^{\alpha p-p}&\int_s^T|g^\e(T-v)-g(T-v)|^pdvds\\
   &\leq \int_0^T\frac{|\beta|^p\e^p}{1+(\beta-1)p}(T-s)^{1+(\alpha + \beta-2)p}\\
   &\hspace{0.5cm} +\Bigl(\frac{1}{1+\beta p} - \frac{|\beta|^p}{1+(\beta-1) p}\Bigr) \e^{1+\beta p}(T-s)^{\alpha p - p}ds,
  \end{align*}
  which converges to zero if $2+(\alpha+\beta-2)p>0$.\\
  The same estimate can be applied for $0<\beta<1$. 
  In this case we find the same condition as above. 
  As far as the case $\beta>1$ is concerned, similar reasonings as in Example~\ref{ex:Y} can be done, leading to convergence if $1+ (\alpha -1)p  > 0$.
  These conditions are the same as in Example \ref{ex:appr}.
  \item 
  Consider $\beta <0$.
  Observe that $   \bar g^\e(x):=(x+\e)^\beta-x^\beta$ is negative and increasing, hence
  $$
  \vert \bar g^\e(T-v) - \bar g^\e(s-v) \vert = \vert \bar g^\e (s-v) \vert.
  $$
  Then we have
  \begin{align*}
  &\int_0^T\int_0^s\frac{|(g^\e(T^--v)-g(T^--v))-(g^\e(s-v)-g(s-v))|^p}{(T-s)^{p-\alpha p}}dvds\\
   &\leq\int_0^T(T-s)^{\alpha p -p}\int_0^s\left|(s+\e-v)^{\beta}-(s-v)^{\beta}\right|^pdvds\\
   &\leq\int_0^T(T-s)^{\alpha p -p}\Bigg(\int_0^{s-\e}|\beta|^p\e^p(s-v)^{(\beta-1)p}dv +\int_{s-\e}^s(s-v)^{\beta p}dv\Bigg)ds\\
   &\leq\int_0^T(T-s)^{\alpha p -p}\Bigg(\frac{|\beta|^p\e^p}{1+(\beta-1)p}(s^{1+(\beta-1)p}-\e^{1+(\beta-1)p})\\
   &\hspace{3.5cm} +\frac{1}{1+\beta p}\e^{1+\beta p}\Bigg)ds,
  \end{align*}
  where we have applied the same argument as in Example \ref{ex:Y} based on \eqref{N-1} and \eqref{*}.
  Here we require that $2+(\beta-1)p>0$ and $1+\alpha p -p>0$.\\
  With similar arguments, in the case $\beta>0$ it can be shown that the required conditions are $1+\alpha p - p > 0$, if $\beta >1$, and $1+(\beta-1)p >0$, $1+(\alpha -1)p >0$, if $\beta \in (0,1)$.
  These conditions are already guaranteed by those in Example \ref{ex:appr}.
  \item 
 Let $\beta <0$.
 For $\e$ small we split the integrals and apply \eqref{N-1}, \eqref{*}, then we obtain
  \begin{align*}
  \int_0^T&\int_s^T(u-s)^{\alpha p-2p}\int_s^u|g^\e(u-v)-g(u-v)|^pdvduds\\
  &\leq \int_0^T\int_s^{s+\e}(u-s)^{\alpha p-2p}\int_s^u(u-v)^{\beta p}dvduds\\
  & \quad + \int_0^T\int_{s+\e}^T(u-s)^{\alpha p-2p}\Bigg(\int_s^{u-\e}\e^p|\beta|^p(u-v)^{(\beta -1)p}dv\Bigg.\\
  &\hspace{4.5cm} \Bigg.+ \int_{u-\e}^u (u-v)^{\beta p}dv\Bigg)duds\\
  &= \int_0^T\int_s^{s+\e}\frac{(u-s)^{1+\beta p +\alpha p-2p}}{1+\beta p}duds\\
  & \quad + \int_0^T\int_{s+\e}^T(u-s)^{\alpha p-2p}\Bigg(\frac{\e^p|\beta|^p}{1+(\beta -1)p}(u-s)^{1+(\beta -1)p}\Bigg.\\
  &\hspace{3.5cm} + \Bigg.\Bigl( \frac{1}{1+\beta p} - \frac{|\beta|^p}{1+(\beta -1)p}\Bigr)\frac{\e^{1+\beta p}}{1+\beta p}\Bigg)duds\\
  &= \int_0^T \Bigl(C_1(\beta,\alpha,p)\e^{2+\beta p +\alpha p-2p}\Bigr.\\
  & \quad + C_2(\beta,\alpha,p)\e^p((T-s)^{2+\beta p + \alpha p -3p} - \e^{2+\beta p + \alpha p -3p})\\
  & \quad + \Bigl. C_3(\beta,\alpha,p)\e^{1+\beta p}((T-s)^{1+\alpha p -2p} - \e^{1+\alpha p -2p})\Bigr)ds.
  \end{align*}
  Taking into account the study in Example \ref{ex:appr}, we see that the convergence of the expression above is given when $3+(\alpha+\beta-3)p>0$.\\
 Let $\beta>0$. By application of the monotonicity of $\bar g^\e = g^\e-g$ and \eqref{N-1}, we can see that the convergence is obtained when $2+\alpha p - 2 p > 0$.
  \item 
  Let $\beta <0$. We split the integration range in disjoint intervals and study them separately, in a similar fashion as in Example \ref{ex:appr}:
  \begin{align*}
  \int_0^T&\int_s^T\int_0^s\frac{|(g^\e(u-v)-g(u-v))-(g^\e(s-v)-g(s-v))|^p}{(u-s)^{2p-\alpha p}}dvduds\\
  &\leq\int_0^{\e} \Bigg(\int_s^{2s}(u-s)^{\alpha p-2p}\int_0^s\left|(u-v)^{\beta}-(s-v)^{\beta}\right|^pdvdu \Bigg.\\
  &\hspace{1cm} \Bigg. +\int_{2s}^T(u-s)^{\alpha p-2p}\int_0^s\left|(u-v)^{\beta}-(s-v)^{\beta}\right|^pdvdu \Bigg)ds\\
  &\quad +\int_{\e}^T\Bigg(\int_s^{s+\e}(u-s)^{\alpha p-2p}\int_0^s\left|(u-v)^{\beta}-(s-v)^{\beta}\right|^pdvdu \Bigg.\\
  &\hspace{1.5cm} \Bigg. +\int_{s+\e}^T(u-s)^{\alpha p-2p}\int_0^s\left|(s+\e-v)^{\beta}-(s-v)^{\beta}\right|^pdvdu \Bigg)ds\\
  &=: A_1 + A_2 + B_1 + B_2.
  \end{align*}
  By application of \eqref{N-1}, \eqref{*} we obtain
  \begin{align*}
  A_1 &= \int_0^{\e}\int_s^{2s}(u-s)^{\alpha p-2p}\int_0^s\left|(u-v)^{\beta}-(s-v)^{\beta}\right|^pdvduds\\
  &\leq \int_0^{\e}\int_s^{2s}(u-s)^{\alpha p-2p}\Bigg(\int_0^{2s-u}|\beta|^p(u-s)^p(s-v)^{(\beta-1)p}dv\Bigg. \\
  &\hspace{4cm} \Bigg. + \int_{2s-u}^s (s-v)^{\beta p}dv\Bigg)duds\\
  &\leq \int_0^{\e}\int_s^{2s}\Bigg(\frac{|\beta|^p}{1+(\beta-1)p}s^{1+(\beta-1)p}(u-s)^{\alpha p-p}\Bigg. \\
  &\hspace{2cm} \Bigg. + \Bigl(\frac{1}{1+\beta p} - \frac{|\beta|^p}{1+(\beta-1)p}\Bigr) (u-s)^{1+\beta p +\alpha p -2p}\Bigg)duds\\
  &= \int_0^{\e} C(\beta,\alpha,p)s^{2+\beta p +\alpha p -2p} ds,
  \end{align*}
  which converges to $0$ if we require $2+(\alpha+\beta-2)p>0$. 
  The next summand is given by
  \begin{align*}
  A_2 &= \int_0^{\e}\int_{2s}^T(u-s)^{\alpha p-2p}\int_0^s\left|(u-v)^{\beta}-(s-v)^{\beta}\right|^pdvduds\\
  &\leq \int_0^{\e}\int_{2s}^T(u-s)^{\alpha p-2p}\int_0^s(s-v)^{\beta p}dvduds \\
  &= \int_0^{\e}\frac{s^{1+\beta p}}{1+\beta p}\frac{1}{1+\alpha p -2p}\Bigl((T-s)^{1+\alpha p -2p} - s^{1+\alpha p -2p}\Bigr)ds,
  \end{align*}
where the last integral converges if $3+(\alpha+\beta-2)p>0$. 
We proceed to the next summand:
  \begin{align*}
  B_1 &= \int_{\e}^T\int_s^{s+\e}(u-s)^{\alpha p-2p}\int_0^s\left|(u-v)^{\beta}-(s-v)^{\beta}\right|^pdvduds\\
  &\leq \int_{\e}^T\int_s^{s+\e}(u-s)^{\alpha p-2p}\Bigg(\int_0^{2s-u}|\beta|^p(u-s)^p(s-v)^{(\beta-1)p}dv\Bigg. \\
  &\hspace{4cm} \Bigg. + \int_{2s-u}^s (s-v)^{\beta p}dv\Bigg)duds\\
  &\leq \int_{\e}^T\int_s^{s+\e}\Bigg(\frac{|\beta|^p}{1+(\beta-1)p}s^{1+(\beta-1)p}(u-s)^{\alpha p-p}\Bigg. \\
  &\hspace{2cm} \Bigg. + \Bigl(\frac{1}{1+\beta p} - \frac{|\beta|^p}{1+(\beta-1)p}\Bigr) (u-s)^{1+\beta p +\alpha p -2p}\Bigg)duds\\
  &= \int_{\e}^T \Bigl(C_1(\beta,\alpha,p)\e^{1+\alpha p -p}s^{1+(\beta-1)p} + C_2(\beta,\alpha,p) \e^{2+\beta p +\alpha p -2p}\Bigr) ds.
  \end{align*}
Its convergence is given if $2+(\alpha+\beta-2)p>0$. Then considering the last piece we get
  \begin{align*}
  B_2 &= \int_{\e}^T\int_{s+\e}^T(u-s)^{\alpha p-2p}\int_0^s\left|(s+\e-v)^{\beta}-(s-v)^{\beta}\right|^pdvduds\\
  &\leq \int_{\e}^T\int_{s+\e}^T(u-s)^{\alpha p-2p}\Bigg(\int_0^{s-\e}|\beta|^p\e^p(s-v)^{(\beta-1)p}dv\Bigg. \\
  &\hspace{4cm} \Bigg. + \int_{s-\e}^s (s-v)^{\beta p}dv\Bigg)duds\\
  &\leq \int_{\e}^T\int_{s+\e}^T(u-s)^{\alpha p-2p}\Bigg(\frac{|\beta|^p\e^p}{1+(\beta-1)p}s^{1+(\beta-1)p}\Bigg. \\
  &\hspace{2cm} \Bigg. + \Bigl(\frac{1}{1+\beta p} - \frac{|\beta|^p}{1+(\beta-1)p}\Bigr)\e^{1+\beta p}\Bigg)duds\\
  &= \int_{\e}^T \Bigg(\frac{|\beta|^p\e^p}{1+(\beta-1)p}s^{1+(\beta-1)p} + \Bigl(\frac{1}{1+\beta p} - \frac{|\beta|^p}{1+(\beta-1)p}\Bigr)\e^{1+\beta p}\Bigg) \\
  &\hspace{2cm} \frac{1}{1+\alpha p -2p}\Bigl((T-s)^{1+\alpha p -2p} - \e^{1+\alpha p -2p}\Bigr)ds.
  \end{align*}
If $2+(\alpha+\beta-2)p>0$, we have convergence of the last integral.\\
  For $\beta$ positive, with similar reasoning as above, we find that convergence is guaranteed when $2+(\alpha-2)p>0$, $1+(\beta-1)p>0$, when $0<\beta<1$, and $1+\alpha p -p>0$, when $\beta>1$.
 \end{enumerate}
 Summarising, for the the driving L\'evy process $L$ with characteristic triplet $(0,0,\nu)$ and symmetric measure $\nu$ such that $\int_\mathbb{R} \vert z \vert^p \nu(dz) < \infty$, we have considered the integrals $\int_0^T X(t) dY^\e(t)$ with $Y^\e(t) = \int_0^T (t-s+\e)^\beta dL(s) $ and the integral $\int_0^T X(t) dY(t)$ with  $Y^(t) = \int_0^T (t-s)^\beta dL(s) $, see Definition \ref{stoGLS}. Then we have shown that there is convergence of the integrals in $L_1(\Omega)$, according to Theorem \ref{prop:approx2}, if one of the following conditions is satisfied:
 \begin{itemize}
 \item 
 $\beta>1, \alpha \in (0,1), p \geq 1, 2+(\alpha-2)p >0$
 \item
  $\beta \in (0,1), \alpha \in (0,1), p \geq 1, 2+(\alpha-2)p >0, 1+ (\beta-1)p>0$
 \item
   $\beta \in (-1,0), \alpha \in (0,1), p \geq 1, 3+(\alpha+\beta-3)p >0$.
 \end{itemize}
\end{example}

\vspace{4mm}
We conclude this section with a numerical example of a Gamma-Volterra process as integrator driven by a pure jump L\'evy process with infinite activity. The parameters are taken according to the sufficient conditions found in Example \ref{ex:appr} and Example \ref{ex:general}. We illustrate the use of the approximation result for simulation purposes. Approximating the non-semimartingale by the corresponding semimartingale as per Theorem \ref{prop:approx2}, we can then exploit the connection of the fractional integral with the It\^o integral that we have remarked earlier.

\vspace{2mm}
\begin{example}\label{ex:complete}
 \textbf{Numerical example of the approximation of an integral with Volterra driver.} In light of Example \ref{ex:general}, for illustration, set the parameters in \eqref{eq:keyex} to be $\lambda=0,\;\beta=-1/16$ and $p=9/8$, i.e.
 \[
  Y(t)=\int_0^t(t-s)^{-1/16}\;dL(s),
 \]
 and let $L$ be a symmetric tempered stable L\'evy process with L\'evy measure $ \nu(dz) = C\frac{e^{-\gamma |z|}}{|z|^{1+\alpha_L}}dz $ ($\alpha_L<2, \gamma>0$). As illustration, choose $\gamma=10$, $\alpha_L=1/2$.
 Then we see that the $p$-th moment of the L\'evy measure is finite
 \[
  \int_{\mathbb{R}} |z|^p\nu(dz) = 2C \int_0^{\infty} z^{(p-\alpha_L)-1} e^{-\gamma z}dz<\infty,
 \]
since we can obtain the Gamma function by a change of variable and also $p-\alpha_L>0$.

From Example \ref{coro:semi} we know that $Y$ is not a semimartingale. From Example \ref{ex:appr}, taking $\alpha=2/5$, $q=9$, we know that $Y$ is an appropriate $(p,\alpha)-$integrator for any $X\in\E\mathcal{D}^+_q(\alpha,T)$, and that \( g \) and \( g^\varepsilon \) satisfy the convergence conditions of Theorem \ref{prop:approx2}, respectively. The fact that \( Y^\varepsilon \) is a semimartingale is deduced from \eqref{eq:semicond} in Theorem \ref{thm:basse}: For all $\e>0$,
 \begin{align*}
  \int_0^t&\int_{[-1,1]}|z(g^\e)'(s)|^2\wedge|z(g^\e)'(s)|\nu(dz)ds\\
  &=|\beta|\int_0^t\int_{[-1,1]}|z(s+\e)^{\beta-1}|^2\wedge|z(s+\e)^{\beta-1}|\nu(dz)ds\\
  &\leq |\beta|\e^{2(\beta-1)}t\int_{[-1,1]}|z|^2\wedge|z|\nu(dz)\\
  &\leq |\beta|\e^{2(\beta-1)}T\int_{[-1,1]}|z|^2\nu(dz)<\infty.
 \end{align*}
Thus from Example \ref{ex:general}, for these values of the parameters, we have the convergence $\int XdY^\e\longrightarrow \int XdY$ in $L_1(\Omega)$.

\vspace{2mm}
To complete the example, we consider two integrands. First we take $X(t)=t$. We see that $X\in\E\mathcal{D}^+_q(\alpha,T)$, in fact
\begin{align*}
\int_0^T\E|\left(\mathcal{D}^\alpha_{0^+}X\right)(s)|^q ds &= \int_0^T\E \left| \frac{X(s)}{s^\alpha} + \alpha \int_0^s \frac{X(s) - X(u)}{(s-u)^{1+\alpha}} du \right|^q ds\\
&= \int_0^T \left| s^{1-\alpha} + \alpha \frac{s^{1-\alpha}}{1-\alpha} \right|^q ds\\
&= \int_0^T \left(1 + \frac{\alpha}{1-\alpha}\right) s^{q(1-\alpha)} ds < \infty.
\end{align*}
The second integrand is given by $X(t)=B(t)$, where $B(t)$ is a Brownian motion. In this case the fractional derivative is Gaussian and, to show that $X\in\E\mathcal{D}^+_q(\alpha,T)$, we first find the second moment:
\begin{align*}
\E|\left(\mathcal{D}^\alpha_{0^+}B\right)(s)|^2 &= \frac{\E[B^2(s)]}{s^{2\alpha}} + \alpha^2 \! \int_0^s \int_0^s \frac{\E[(B(s) - B(v))(B(s) - B(u))]}{(s-v)^{1+\alpha} (s-u)^{1+\alpha}} dudv\\
& \quad + 2\alpha \int_0^s \frac{\E[B(s)(B(s) - B(v))]}{s^\alpha (s-v)^{1+\alpha}}dv\\
&= I_1 + I_2 + I_3.
\end{align*}
We see that $I_1 = s^{1-2\alpha}$, while
\begin{align*}
I_2 &= \alpha^2 \int_0^s \Bigl( \int_0^v \frac{1}{(s-v)^\alpha (s-u)^{1+\alpha}} du + \int_v^s \frac{1}{(s-v)^{1+\alpha} (s-u)^\alpha} du \Bigr) dv\\
&= \alpha^2 \int_0^s \Bigl( \frac{1}{(s-v)^\alpha}\frac{((s-v)^{-\alpha} - s^{-\alpha})}{\alpha} + \frac{1}{(s-v)^{1+\alpha}}\frac{(s-v)^{1-\alpha}}{1-\alpha} \Bigr) dv\\
&= \alpha^2 \int_0^s \Bigl( \frac{1}{\alpha(1-\alpha)} (s-v)^{-2\alpha} dv - \frac{s^{-\alpha}}{\alpha} (s-v)^{-\alpha} \Bigr) dv\\
&= \frac{2\alpha^2}{(1-\alpha)(1-2\alpha)} s^{1-2\alpha}.
\end{align*}
Similarly $I_3 = \frac{2\alpha}{1-\alpha} s^{1-2\alpha}$. Hence we find that
\begin{equation*}
\int_0^T\E|\left(\mathcal{D}^\alpha_{0^+}B\right)(s)|^q ds = \int_0^T \frac{2^{\frac{q}{2}}\Gamma(\frac{q+1}{2})}{\sqrt{\pi}} C(\alpha)^{\frac{q}{2}}s^{(1-2\alpha)\frac{q}{2}} ds < \infty,
\end{equation*}
and $X\in\E\mathcal{D}^+_q(\alpha,T)$.

\vspace{2mm}
\noindent
Figure \ref{fig:figure1}, Figure \ref{fig:figure2} and Figure \ref{fig:figure3} present a simulation of the processes described above: $L, Y^\epsilon$, and $\int XdY^\epsilon$, with $g^\epsilon(t-s)=(t+\epsilon-s)^{-1/16}$ and $\epsilon=10^{-10}$. For the simulation of a sample path of the tempered stable L\'evy process in Figure \ref{fig:figure1} we used the series representation by Rosi{\'n}ski \cite{rosi}. The simulation of the Volterra process $Y^\epsilon$ is then obtained by means of a classical numerical integration with an Euler scheme, see e.g. \cite{KP}, by using the sample path of $L$.

\begin{figure}
  \centering
    \includegraphics[width=0.8\textwidth]{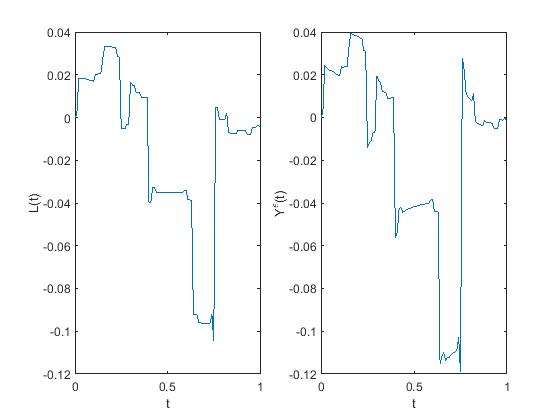}
    \caption{
    The graph on the left show a simulated path of the tempered stable process $L$, while on the right a path the Gamma-Volterra process $Y$.}
    \label{fig:figure1}
\end{figure}

Using the same approach we also simulated the integrals in the case of $X(t)=t$ in Figure \ref{fig:figure2} and $X(t)=B(t)$ in Figure \ref{fig:figure3}. In particular, since the integrands are clearly adapted and c\`agl\`ad, the integral $ \int X dY^\epsilon $, $ \epsilon>0 $, is well defined as an It\^o integral. This is exploited in our simulation, which is again obtain by numerical integration, with the same value for $\epsilon$. Theorem \ref{prop:approx2} will then guarantee that the simulations convergence in $L_1$ to the pathwise fractional integral $\int X dY$, approximated in this way by It\^o integrals.

\begin{figure}
  \centering
    \includegraphics[width=0.8\textwidth]{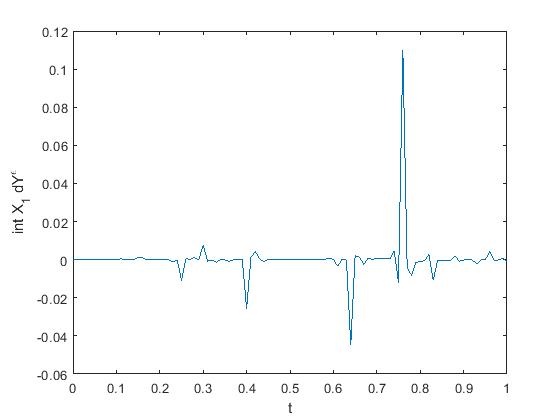}
    \caption{
    A path the simulated fractional integral with $X_1(t) = t$.}
    \label{fig:figure2}
\end{figure}

\begin{figure}
  \centering
    \includegraphics[width=0.8\textwidth]{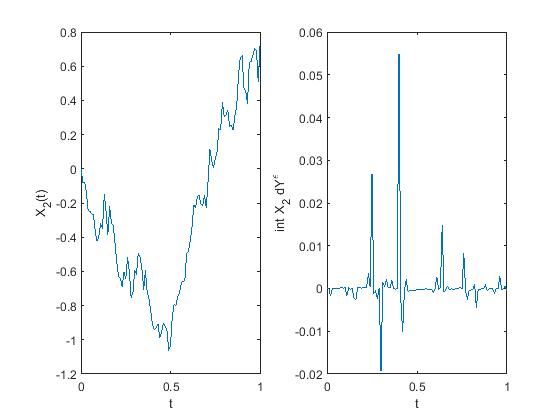}
    \caption{
    On the left a simulated path of the Brownian motion $B$. On the right the simulated fractional integral with $X_2(t) = B(t)$.}
    \label{fig:figure3}
\end{figure}
\end{example}

\section{Conclusion}
In the framework of fractional integrals, we have studied an approximation of the fractional stochastic integral $\int XdY$ for non-semimartingale integrators $Y$ by a sequence of integrals $(\int XdY^\e)_\e$ with semimartingale integrators $Y^\e$. No filtration structure is needed on the integrand $X$ for the definition of the integral or for the convergence. However, in the case when $X$ is an adapted, c\`agl\`ad process the integral $\int XdY^\e$ agrees with the usual It\^o-integral, thus $\int XdY$ can be approximated by a sequence of It\^o-integrals. 
As illustration we have specialised our results to the case of Gamma-Volterra processes driven by L\'evy processes, which is a family of models largely used in applications and of recent attention in energy finance and biological modelling.
We have shown how the approximation procedures proposed is used for computational purposes by examples.

\vspace{5mm}
\noindent
{\bf Acknowledgement.}
The research leading to these results is within the project STORM: Stochastics for Time-Space Risk Models, receiving funding from the Research Council of Norway (RCN). Project number: 274410.


\end{document}